\documentclass[a4paper]{amsart} 

\usepackage{a4wide}
\usepackage[latin1]{inputenc}
\usepackage{amssymb,amsmath,amsthm}
\usepackage{url}
\usepackage[bookmarks=false,pdfborder={0 0 0.05}]{hyperref}
\usepackage{color}
\usepackage{tikz}
\usetikzlibrary{calc,through,backgrounds}
\usepackage{tikz-qtree}
\usepackage{enumerate}

\newtheorem{theorem}{Theorem}[section]
\newtheorem{proposition}[theorem]{Proposition}
\newtheorem{corollary}[theorem]{Corollary}

\theoremstyle{definition}

\newtheorem*{remark}{Remark}  

\newcommand{\defn}[1]{\emph{\color{red} #1}} 
\def\areaseq{\operatorname{areaseq}}
\def\area{\operatorname{area}}
\def\rises{\operatorname{rises}}
\def\returns{\operatorname{returns}}
\def\Cat{\operatorname{Cat}}
\def\NN{\mathbb{N}}

\begin{document}
  \title{On a new collection of words in the Catalan family}
  
  \author[Christian Stump]{Christian Stump$^\dagger$}
  \address[C.~Stump]{Institut f\"ur Mathematik, Freie Universit\"at Berlin, Germany}
  \email{christian.stump@fu-berlin.de}
  \urladdr{http://homepage.univie.ac.at/christian.stump/}
  \thanks{$^\dagger$Research supported by the German Research Foundation DFG, grant STU 563/2-1 ``Coxeter-Catalan combinatorics''}

  \begin{abstract}
    In this note, we provide a bijection between a new collection of words on nonnegative integers of length~$n$ and Dyck paths of length $2n-2$, thus proving that this collection belongs to the Catalan family.
    The surprising key step in this bijection is the zeta map which is an important map in the study of $q,t$-Catalan numbers.
    Finally we discuss an alternative approach to this new collection of words using two statistics on planted trees that turn out to be closely related to the Tutte polynomial on the Catalan matroid.
  \end{abstract}

  \maketitle

  \section{Introduction}

  M.~Albert, N.~Ru\v{s}kuc, and V.~Vatter~\cite{ARV2014} recently introduced the following collection of words on nonnegative integers, and asked on MathOverflow~\cite{Vat2013} for ``a nice bijection between these words and any family of classical Catalan objects such as Dyck paths or noncrossing partitions''.
  For a positive integer~$n$, let $\mathcal{L}_n$ denote the set of all words $a = (a_1,\ldots,a_n)$ of~$n$ nonnegative integers such that
  \begin{enumerate}
    \item[(A1)] $a_{i+1} \geq a_i-1$ for $1 \leq i < n$,
    \item[(A2)] if $a_i = k > 0$ with $i$ minimal, then there exist $i_1 < i < i_2$ such that $a_{i_1} = a_{i_2} = k-1$.
  \end{enumerate}
  The first property says that such sequences do not have drops greater than one, while the second says that the leftmost occurrence of $k$ in $a$ has a $k-1$ somewhere to its left and somewhere to its right.
  We refer to these two properties as \defn{Property~$(A)$}.
  The only word of length $2$ with these properties is
  $$00.$$
  There are $2$ words of length $3$ given by
  $$000,010.$$
  For length $4$, there are $5$ such words,
  $$0000,0010,0100,0101,0110,$$
  and for length $5$ there are $14$,
  $$
  00000,00010,00100,00101,00110,01000,01001,
  01010,01011,01021,01100,01101,01110,01210.
  $$
  V.~Vatter asked in~\cite{Vat2013} for a bijective proof that this collection of words of length~$n$ is counted by the \defn{$(n-1)$-st Catalan number} $\Cat_{n-1} = \frac{1}{n}\binom{2n-2}{n-1}$.
  In this note I provide such a bijective proof, recording my answer given at~\cite{Vat2013}.

  \medskip

  \defn{Dyck paths} are lattice paths from $(0,0)$ to $(n,n)$ that never go below the diagonal $x=y$.
  We encode a Dyck path~$D$ as a sequence of~$n$ north and~$n$ east steps such that every prefix of~$D$ does not contain more east steps than north steps, and denote all Dyck paths of length $2n$ by $\mathcal{D}_n$.
  For example, for~$n=3$, there are $5$ Dyck paths of length~$6$, namely
  $$NNNEEE, NNENEE, NNEENE, NENNEE, NENENE.$$

  \begin{theorem}
    There is an explicit bijection $\mathcal{L}_n\ \tilde\longrightarrow\ \mathcal{D}_{n-1}$.
  \end{theorem}
  The procedure in the following section yields the explicit bijection proposed in the theorem.
  We start with defining a few statistics on Dyck paths that will be used in this note.
  To this end, let~$D$ be a Dyck path of length~$2n$.
  \begin{itemize}
    \item The \defn{area sequence} $\areaseq(D) = (a_1,\ldots,a_n)$ is given by
      $a_i = i - x_i - 1$ where $x_i$ is the number of east steps before the $i$\textsuperscript{th} north step in~$D$.
      The $5$ Dyck paths above thus get the $5$ area sequences
      $$(0,1,2), (0,1,1), (0,1,0), (0,0,1), (0,0,0).$$

    \item The \defn{area statistic} $\area(D)$ is given by the sum of the entries in the area sequence,
    \item the \defn{number of initial rises} $\rises(D)$ is given by the number of north steps in~$D$ before the first east step.
    \item the \defn{number of returns} $\returns(D)$ is the number of returns of~$D$ to the main diagonal.
    In symbols, $\returns(D) = \#\big\{ 1 \leq i \leq n \ :\ \areaseq(D)_i = 0 \big\}$.
    Finally, a return is called an \defn{inner touch point} if it is not the final return to the main diagonal in the point $(n,n)$.
    Thus, the number of inner touch points is given by $\#\big\{ i \ :\ \areaseq(D)_i = 0 \big\}-1$.
  \end{itemize}

  \section{The procedure}

  For the reader's convenience, we provided a Sage worksheet implementing each step in the construction at \url{http://sage.lacim.uqam.ca/home/pub/33/}.

  \bigskip

  Let $a = (a_1,\ldots,a_n)$ be a sequence of nonnegative integers.
  It satisfies \defn{Property~$(B)$} if
  \begin{enumerate}
    \item[(B1)] $a_{i+1} \leq a_i+1$,
    \item[(B2)] if $a_i = k > 0$ with $i$ minimal, then there exist $i_1 < i < i_2$ such that $a_{i_1} = a_{i_2} = k-1$.
  \end{enumerate}

  Interchanging neighbors that do not satisfy Property~$(B1)$ does not interfere with Property $(A2) = (B2)$ and thus provides a bijection between sequences with Property~$(A)$ and those with Property~$(B)$.
  For example, there are eight sequences of length $6$ satisfying Property~$(A)$ that do not satisfy Property~$(B1)$,
  $$001021, 011021, 010021, 010210, 010211, 010212, 012102, 010221.$$
  Interchanges $0$'s and $2$'s where necessary then yields
  $$001201, 011201, 012001, 012010, 012011, 012012, 012120, 012201.$$

  \medskip

  Next, we say that $a = (a_1,\ldots,a_n)$ satisfies \defn{Property~$(C)$} if
  \begin{enumerate}
    \item[(C1)] $a_1 = 0$
    \item[(C2)] $a_{i+1} \leq a_i+1$,
    \item[(C3)] if $a_i = k > 0$ with $i$ minimal, then there exist $i < i_2$ such that $w_{i_2} = k-1$.
  \end{enumerate}
  Properties~$(B)$ and~$(C)$ are equivalent since $(B2)$ implies that $a_1 = 0$. Together with $(B1) = (C2)$ this then implies that every $a_i=k$ in the sequence $a$ has a $k-1$ somewhere to its left, and we can drop this part of $(B2)$ to obtain $(C3)$.

  \medskip
  
  It is now well known that the map sending a Dyck path $D \in \mathcal{D}_n$ to its area sequence $\areaseq(D)$ is a bijection between $\mathcal{D}_n$ and sequences satisfying Properties~$(C1)$ and~$(C2)$.
  We thus say that $a = (a_1,\ldots,a_n)$ satisfies \defn{Property~$(D)$} if it satisfies Properties~$(C1)$ and~$(C2)$, and call such sequences \defn{area sequences}.

  \medskip
  
  Since Property~$(C)$ is strictly stronger than Property~$(D)$, we have reached an embedding of sequences of length $n$ with Property~$(A)$ into Dyck paths of length $2n$.
  Next, we apply the \defn{zeta map} $\zeta: \mathcal{D}_n \longrightarrow \mathcal{D}_n$, as studied for example in~\cite[page~50]{Hag2008}.
  This map is defined by given a sequence $a = (a_1,\ldots,a_n)$ satisfying Property~$(D)$, it returns a Dyck path as follows:
  \begin{itemize}
    \item Build an intermediate Dyck path (the \defn{bounce path}) consisting of $d_1$ north steps, followed by $d_1$ east steps, followed by $d_2$ north steps and $d_2$ east steps, and so on, where $d_i$ is the number of $i-1$'s within $a$. For example, given $a = (0,1,2,2,2,3,1,2)$, we build the path $NE\ NNEE\ NNNNEEEE\ NE$ (this is the dashed path in~\cite[Figure~3]{Hag2008}).

    \item Next, the rectangles between two consecutive peaks of the bounce path are filled. Observe that such the rectangle between the $k$-th and the $(k+1)$-st peak must be filled by $d_k$ east steps and $d_{k+1}$ north steps. In the above example, the rectangle between the second and the third peak must be filled by $2$ east and $4$ north steps, the $2$ being the number of $1$'s in $a$, and $4$ being the number of $2$'s. To fill such a rectangle, scan through the sequence $a$ from left to right, and add east or north steps whenever you see a $k-1$ or $k$, respectively. So to fill the $2 \times 4$ rectangle, we look for $1$'s and $2$'s in the sequence and see $122212$, so this rectangle gets filled with $ENNNEN$.

    \item This completes the zeta map, and the path we obtain in the example is then given by $N\ ENN\ ENNNEN\ EEENE\ E$.
  \end{itemize}

  This zeta map has obtained quite some attention in the past 10 years in the context of the (still open) problem to combinatorially understanding the symmetry of the $q,t$-Catalan numbers.
  It was constructed for the following two remarkable properties (which do not play any significant role in the present context):

  \begin{itemize}

    \item It sends the \defn{dinv statistic} given by the number of pairs $k<\ell$ with $a_k-a_\ell \in \{0,1\} $ to the area statistic.

    \item It sends the area statistic to the \defn{bounce statistic} given by the sum of the weighted bounce points $\sum_i id_i$ where the $d_i$'s are the inner touch points of the bounce path as given in the first step of the definition of the zeta map.

  \end{itemize}

  The reason why this map is the key to provide a bijection for the new collection of words considered in this note is given by the following two further properties, which are both direct consequences of the definition.
  Nevertheless, to the best of my knowledge, they have not been used in the literature before.
  For $D \in \mathcal{D}_n$, we have that
  \begin{enumerate}[(i)]

    \item the zeta map sends the number of $0$'s in $\areaseq(D)$ to the number of north steps before the first east step in $\zeta(D)$, and 

    \item it sends the number of $i$'s for which the last occurrence of $i$ in $\areaseq(D)$ is left of the first occurrence of the first $i+1$ to the number of inner touch points of $\zeta(D)$. \label{eq:zeta2}

  \end{enumerate}

  Observe that~\eqref{eq:zeta2} can be reformulated in the way that $\areaseq(D) = (a_1,\ldots,a_n)$ satisfies Property~$(C3)$ if and only if $\zeta(D)$ leaves the diagonal in the very beginning and only returns in the very end, and nowhere in between.
  Thus, stripping off the first north and the last east step from $\zeta(D)$ yields a Dyck path of length $2n-2$, and we finally completed the proposed bijection.

  \section{A related bistatistic on planted trees and the Catalan matroid}

  A \defn{planted tree}\footnote{In~\cite{Spe2013}, such trees are called \emph{rooted planar trees}. I use \emph{planted trees} here since the order of children is not only given cyclically, but linearly. I thank Christian Krattenthaler for bringing this term to my attention.}
  is a rooted tree for which all children of a vertex come in a given linear order.
  The following gives a well known bijection between planted trees on vertices $\{ 0,\ldots,n\}$ and area sequences of length~$n$.
  Start with an area sequence $a = (a_1,\ldots,a_n)$ and associate with it a planted tree by saying that the vertex~$i$ for $i>0$ lives in generation $a_i$, and the parent of~$i$ is the biggest $j < i$ for which $a_j = a_i - 1$.
  Finally, add a unique root in generation $-1$.
  The inverse map is given by clockwise traveling around the planted tree starting from the root, and recording a north step whenever traveling an edge away from the root, and an east step when traveling towards the root.
  
  Following the notation in~\cite{Spe2013}, we think of the vertices of such a tree as members of an asexually reproducing species, and therefore use language like ``child'', ``parent'', ``generation'', and consider the ordering of the vertices in a given generation as their birth order.
  A vertex~$v$ is called \defn{crucial} if~$v$ is the youngest member of its generation, all the other members of that generation are childless while~$v$ has children.
  Observe in particular that for $n\geq 2$, the root is always crucial.
  The reason for considering crucial vertices is that given a Dyck path $D \in \mathcal{D}_n$ with area sequence $\areaseq(D) = (a_1,\ldots,a_n)$ and corresponding tree~$T$, then $k>0$ violates Property~$(C3)$ for $\areaseq(D)$ if and only if the youngest member of generation~$k-1$ is crucial.

  \medskip

  The planted tree corresponding to the sequence $a = (0,1,2,2,2,3,1,2)$ considered above is given by
  
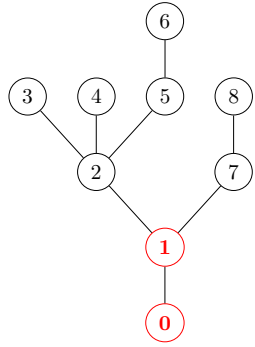
\begin{center}
  \begin{tikzpicture}[grow'=up, scale=0.8, every tree node/.style={draw,circle},
    level distance=1.25cm,sibling distance=.5cm, 
    edge from parent path={(\tikzparentnode) -- (\tikzchildnode)}]
    \Tree
      [.\node[red]{\bf 0};
        [.\node[red]{\bf 1};
          [.2
            [.3 ]
            [.4 ]
            [.5
              [.6 ] ] ]
        [.7 [.8 ] ]
      ] ]
  \end{tikzpicture}
  \end{center}
  Its only crucial vertices are $0$ and $1$.
  The $0$ is the root and as such always crucial, while the crucial vertex in generation~$0$ corresponds to the fact that all $0$'s in $(0,1,2,2,2,3,1,2)$ come before all $1$'s, thus violating~$(C3)$.
  Moreover, the $5$ trees on $4$ vertices are given by
  
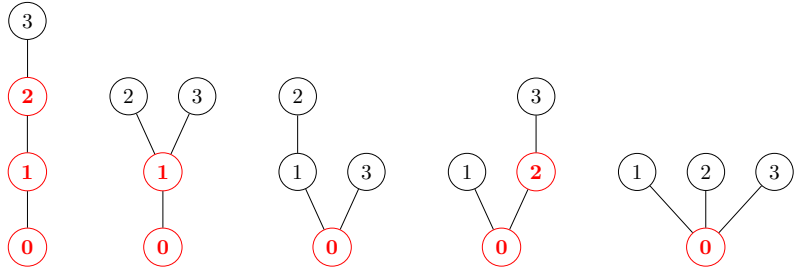
\begin{center}
  \begin{tikzpicture}[grow'=up, scale=0.8, every tree node/.style={draw,circle},
    level distance=1.25cm,sibling distance=.5cm, 
    edge from parent path={(\tikzparentnode) -- (\tikzchildnode)}]
    \Tree
      [.\node[red]{\bf 0};
        [.\node[red]{\bf 1};
          [.\node[red]{\bf 2};
            [.3 ]
            ]
          ]
        ]
      ]
  \end{tikzpicture}
  \qquad
  \begin{tikzpicture}[grow'=up, scale=0.8, every tree node/.style={draw,circle},
    level distance=1.25cm,sibling distance=.5cm, 
    edge from parent path={(\tikzparentnode) -- (\tikzchildnode)}]
    \Tree
      [.\node[red]{\bf 0};
        [.\node[red]{\bf 1};
          [.2 ]
          [.3 ]
        ]
      ]
    ]
  \end{tikzpicture}
  \qquad
  \begin{tikzpicture}[grow'=up, scale=0.8, every tree node/.style={draw,circle},
    level distance=1.25cm,sibling distance=.5cm, 
    edge from parent path={(\tikzparentnode) -- (\tikzchildnode)}]
    \Tree
      [.\node[red]{\bf 0};
        [.1
          [.2 ]
        ]
        [.3 ]
      ]
    ]
  \end{tikzpicture}
  \qquad
  \begin{tikzpicture}[grow'=up, scale=0.8, every tree node/.style={draw,circle},
    level distance=1.25cm,sibling distance=.5cm, 
    edge from parent path={(\tikzparentnode) -- (\tikzchildnode)}]
    \Tree
      [.\node[red]{\bf 0};
        [.1 ]
        [.\node[red]{\bf 2};
          [.3 ]
        ]
      ]
    ]
  \end{tikzpicture}
  \qquad
  \begin{tikzpicture}[grow'=up, scale=0.8, every tree node/.style={draw,circle},
    level distance=1.25cm,sibling distance=.5cm, 
    edge from parent path={(\tikzparentnode) -- (\tikzchildnode)}]
    \Tree
      [.\node[red]{\bf 0};
        [.1 ]
        [.2 ]
        [.3 ]
      ]
    ]
  \end{tikzpicture}
  \end{center}
  \medskip

  Denote by $c(p,q,n)$ the number of planted trees on~$n$ vertices with~$p$ crucial vertices, and where the root has~$q$ children.
  For example, among the previous $5$ planted trees on~$4$ vertices, there is one tree each with $(p,q)$ equal to
  $$(3,1), (2,1), (1,2), (2,2), (1,3).$$
  The following properties of the above bijection are straightforward.
  \begin{proposition}\label{prop:treeprop}
    Let $T$ be a planted tree on $n+1$ vertices and let $D \in \mathcal{D}_n$ be the corresponding Dyck path.
    Then
    \begin{itemize}
      \item the number of children of the root of~$T$ equals the number of $0$'s in $\areaseq(D)$, and
      \item the number of crucial non-root vertices of~$T$ equals the number of indices $i$ for which all $i$'s appear before all $i+1$'s within the area sequence.
    \end{itemize}
    In particular, $\areaseq(D)$ satisfies Property~$(C)$ if and only $0$ is the unique crucial vertex of~$T$.
  \end{proposition}

  D.~Speyer conjectured in~\cite{Spe2013} that for fixed~$n$, the sum $\sum_{p,q \geq 0} c(p,q,n)x^py^q$ equals the Tutte polynomial of the Catalan matroid as defined by F.~Ardila in~\cite{Ard2003}.
  Together with Proposition~\ref{prop:treeprop}, one could then deduce that the number of integer sequences of length~$n$ satisfying Property~$(A)$ are counted by $\sum_{q \geq 0} c(1,q,n) = \sum_{p \geq 0} c(p,1,n)$.
  Since the latter counts planted trees where the root has a unique child, it would then follow that such sequences are counted indeed by the $(n-1)$\textsuperscript{st} Catalan number.

  \medskip

  In the remainder of this section, we show that the zeta map can as well be used to also prove this conjecture.
  We have already seen that
  $$c(1,q,n) = \#\big\{ a \in \NN^n \ :\ a \text{ satisfies Property~$(C)$ and contains exactly $q$ zeroes} \big\}.$$
  Thus, combining the bijection between planted trees and area sequences with the zeta map yields a bijection between planted trees on $n+1$ vertices and Dyck paths of length~$2n$ that sends
  \begin{itemize}
    \item the number of children of the root to the number of initial north steps, and
    \item the number of crucial vertices to the number of returns.
  \end{itemize}
  This implies the following corollary.
  \begin{corollary}
    $c(p,q,n)$ can be reinterpreted in terms of Dyck paths as
    $$c(p,q,n) = \#\big\{ D \in \mathcal{D}_n\ :\ \returns(D) = p,\ \rises(D) = q \big\}.$$
    Moreover, the generating function of $c(p,q,n)$ is given by
    $$\sum_{p,q \geq 0} c(p,q,n) x^p y^q = \sum_{D \in \mathcal{D}_n} x^{\returns(D)} y^{\rises(D)}.$$
  \end{corollary}
  In~\cite{Ard2003}, F.~Ardila introduced and studied the \defn{Catalan matroid}.
  He showed in~\cite[Theorem~3.4]{Ard2003} that the right-hand side of the generating function identity in the previous corollary is actually the \defn{Tutte polynomial} of the Catalan matroid.
  Thus, the connection to the new collection of words considered in this note and its reinterpretation in terms of planted trees yields another combinatorial description of this Tutte polynomial.
  The following corollary can then be derived from~\cite[Theorem~3.6]{Ard2003}. Another proof can be found in~\cite[Theorem~2.1]{ER2013}.
  \begin{corollary}
    $c(p,q,n)$ only depends on the sum $p+q$.
  \end{corollary}
  \begin{proof}
    We here reproduce an elementary argument by D.~Speyer from~\cite{Spe2013}.
    Let~$D$ be a Dyck path of length~$2n$ with $\returns(D) = p$ and $\rises(D) = q$ such that $p \geq 2$.
    Then the following operation on~$D$ yields a Dyck path~$D'$ with $\returns(D') = p-1$ and $\rises(D') = q+1$.
    We can write~$D$ as
    $$N D_1 E\ N D_2 E\ N D_3 E\ \cdots\ N D_p E$$
    such that all~$D_i$'s are itself Dyck paths of smaller length.
    Then~$D'$ is defined as
    $$NN D_1 E\ D_2 E\ N D_3 E\ \cdots\ N D_p E.$$
    Clearly, we have that~$\returns(D') = \returns(D)-1$ and $\rises(D') = \rises(D)+1$, as desired.
    Moreover, it is easy to see that this map is invertible for $p \geq 2$.
  \end{proof}
  \begin{remark}
    The proof of the previous corollary together with the zeta map yields an operator on planted trees with the property that given such a tree~$T$ that has more than one crucial vertex, it constructs a tree~$T'$ having one less crucial vertex and one more child of the root.
    It would be very interesting to find such an operator directly described in planted trees.
    First, this would yield another way of finding a bijection between the collection of words considered here and Dyck paths.
    Second, one could then hope to get an alternative understanding of the zeta map in terms of such trees.
  \end{remark}

  \section{Acknowledgements}

  This note is a long version of my answers to a MathOverflow question asked by V.~Vatter~\cite{Vat2013} and to a follow-up question asked by D.~Speyer~\cite{Spe2013}.
  I thank V.~Vatter for raising the original question, and D.~Speyer for providing further context in his follow-up question.
  Moreover, I thank all other people that contributed to both MathOverflow discussions.

  \bibliographystyle{amsplain} 
  \bibliography{../mrabbrev,../bibliography}
\end{document}